\documentclass[11pt]{amsart}

\usepackage[T1]{fontenc}
\usepackage{lmodern}
\usepackage{microtype}
\usepackage[left=0.8in,right=0.8in,top=1in,bottom=1in]{geometry}
\allowdisplaybreaks[2]
\usepackage{latexsym,amssymb,amsmath}
\usepackage{graphicx,float,color,fancybox,shapepar,setspace,hyperref}

\usepackage{etoolbox}

\makeatletter
\patchcmd{\abstract}
  {\item[\hskip\labelsep\scshape\abstractname.]}
  {\item[]\makebox[\linewidth][c]{\scshape\abstractname}\par\smallskip}
  {}{}
\makeatother

\newtheorem{theorem}{Theorem}[section]
\newtheorem{lemma}[theorem]{Lemma}
\newtheorem{proposition}[theorem]{Proposition}
\newtheorem{corollary}[theorem]{Corollary}
\newtheorem{conjecture}[theorem]{Conjecture}
\newtheorem{problem}[theorem]{Problem}
\theoremstyle{definition}

\newcommand{\eps}{\varepsilon}
\newcommand{\Tcal}{\mathcal T}
\newcommand{\Scal}{\mathcal S}

\begin{document} 
\raggedbottom
\makeatletter

\def\@settitle{%
  \begin{center}%
    \usefont{OT1}{cmr}{m}{n}%
    \fontsize{17.28pt}{21pt}\selectfont
    \@title
  \end{center}%
}

\patchcmd{\@setauthors}
  {\centering\footnotesize}
  {\centering
   \usefont{OT1}{cmr}{m}{n}%
   \fontsize{12pt}{14.4pt}\selectfont}
  {}{}
\patchcmd{\@setauthors}
  {\MakeUppercase{\authors}}
  {\authors}
  {}{}

\def\@seccntformat#1{%
  \protect\textup{%
    \protect\bfseries
    \csname the#1\endcsname
    \protect\quad
  }%
}

\renewcommand{\section}{\@startsection{section}{1}%
  {\z@}%
  {.7\linespacing\@plus\linespacing}%
  {.5\linespacing}%
  {\usefont{OT1}{cmr}{bx}{n}%
   \fontsize{14.4pt}{17pt}\selectfont
   \raggedright}}

\makeatother\title{Gap spectra and densities of slow Fibonacci walks}

\author{Yaping Mao$^{\dagger}$ and Qinghong Zhao$^{\ddagger}$}

\thanks{$^{\dagger}$Academy of Plateau Science and Sustainability and
School of Mathematics and Statistics, Qinghai Normal University,
Xining, Qinghai 810008, China. \texttt{yapingmao@outlook.com}}

\thanks{$^{\ddagger}$Corresponding author: School of Mathematical
Sciences, Huaqiao University, Quanzhou 362021, China.
\texttt{qzhao@hqu.edu.cn}}

\date{}

\begin{abstract}
Let $F_1=F_2=1$ and $F_{t+2}=F_{t+1}+F_t$ for $t\geq1$. For every $n\geq2$, there are unique integers $a,b,t$ such that $n=aF_t+bF_{t-1}$ with $t\geq2$ and $1\leq a\leq b\leq F_t$. The Fibonacci walk with initial pair $(b,a)$ reaches $n$ as late as possible, and the term following $n$ in this walk is $\lfloor\phi n\rfloor$ when $t$ is even and $\lceil\phi n\rceil$ when $t$ is odd, where $\phi=(1+\sqrt5)/2$. Let $D=\{d_1<d_2<\cdots\}$ and $U=\{u_1<u_2<\cdots\}$ be the sets corresponding to even and odd $t$, respectively. For $\ell,m\geq1$, define $D_\ell=\{d_{k+\ell}-d_k:k\geq1\}$, $U_\ell=\{u_{k+\ell}-u_k:k\geq1\}$, $D_\ell(m)=\{d_k:d_{k+\ell}-d_k=m\}$ and $U_\ell(m)=\{u_k:u_{k+\ell}-u_k=m\}$. Chung, Graham and Spiro conjectured that $D_\ell=U_\ell$ for all $\ell$, and asked for the densities of $D_\ell(m)$ and $U_\ell(m)$, especially when $\ell=1$. In this paper, we determine the third and fourth order gap spectra, and show that the conjecture holds for $\ell=3$ but fails for $\ell=4$. We also answer their density question by characterizing when $D_\ell(m)$ and $U_\ell(m)$ have natural densities and proving that their logarithmic densities always exist and are equal. For $\ell=1$, we give the exact logarithmic densities.

\end{abstract}	

\maketitle
\pagestyle{plain}
\thispagestyle{plain}

\section{Introduction}
Let $F_0=0$, $F_1=F_2=1$, and $F_{k+2}=F_{k+1}+F_k$ for $k\geq1$. Given positive integers $a_1$ and $a_2$, the \emph{Fibonacci walk} with initial pair $(a_1,a_2)$ is the sequence $w_k=w_k(a_1,a_2)$ defined by $w_1=a_1,w_2=a_2$ and $w_{k+2}=w_{k+1}+w_k$ for $k\geq1$. Equivalently, $w_k=a_1F_{k-2}+a_2F_{k-1}$ for $k\geq3$. In particular, $w_k(1,1)=F_k$ for $k\geq1$, so the Fibonacci walk with initial pair $(1,1)$ is called the standard Fibonacci sequence. For $n\geq1$, let $s(n;a_1,a_2)$ be the largest integer $s$ such that $w_s(a_1,a_2)=n$, and set $s(n;a_1,a_2)=-\infty$ if no such integer exists. Let $s(n)=\max\limits_{a_1,a_2\geq 1}s(n;a_1,a_2)$. The pair $(a_1,a_2)$ is called \emph{$n$-good} if $s(n;a_1,a_2)=s(n)$, and the corresponding Fibonacci walk is called an \emph{$n$-slow Fibonacci walk}.

The maximal-index problem has been studied in several related forms. In 1991, Cohn~\cite{Cohn} studied recurrence sequences containing a prescribed integer. Further results on representations of prescribed integers at maximal indices in generalized Fibonacci sequences were obtained by Jones and Kiss~\cite{JonesKiss} in 1998 and by Englund and Bicknell-Johnson~\cite{EnglundBicknell} in 2000. In 2020, Chung, Graham, and Spiro~\cite{CGS} studied slow Fibonacci walks, and proved that for every $n\geq2$ there are unique integers $a,b,t$ such that $n=aF_t+bF_{t-1}$ with $t\geq2$ and $1\leq a\leq b\leq F_t$. We call this the canonical representation of $n$. Moreover, the initial pair $(b,a)$ is $n$-good, $s(n)=t+1$, and the term following $n$ in the corresponding walk is $\lfloor\phi n\rfloor$ when $t$ is even and $\lceil\phi n\rceil$ when $t$ is odd, where $\phi=(1+\sqrt5)/2$. Later, Spiro~\cite{Spiro} extended this maximal-index problem to second-order recurrences $w_{k+2}=c_1w_{k+1}+c_2w_k$ with positive integer coefficients $c_1$ and $c_2$. 

For $t\geq2$, let $\Scal_t=\{aF_t+bF_{t-1}:1\leq a\leq b\leq F_t\}$. The uniqueness of the canonical representation shows that the sets $\Scal_t$, $t\geq2$, partition $\{2,3,\ldots\}$. For $n\geq2$, let $\tau(n)$ be the unique integer $t$ such that $n\in\Scal_t$. We call $n$ a \emph{down-integer} if $\tau(n)$ is even and an
\emph{up-integer} if $\tau(n)$ is odd. Let $D=\{n\geq2:\tau(n)\text{ is even}\}=\{d_1<d_2<\cdots\}$ and $U=\{n\geq2:\tau(n)\text{ is odd}\}=\{u_1<u_2<\cdots\}$. The membership of an integer in $D$ or $U$ is called its type. For $\ell\geq1$, define the $\ell$th gap spectra by $D_\ell=\{d_{k+\ell}-d_k:k\geq1\}$ and $U_\ell=\{u_{k+\ell}-u_k:k\geq1\}$. For $\ell,m\geq1$, define the corresponding refined occurrence sets $D_\ell(m)=\{d_k:d_{k+\ell}-d_k=m\}$ and $U_\ell(m)=\{u_k:u_{k+\ell}-u_k=m\}$. Chung, Graham, and Spiro~\cite{CGS} proved that $D_1=U_1=\{1,2,3,5\}$ and
$D_2=U_2=\{2,3,4,5,6,8,10\}$. They further posed the following question: Is $D_3=U_3=\{3,4,5,6,7,8,10,11,13\}$? More generally, they proposed the following conjecture.
\begin{conjecture}[\cite{CGS}]
\label{conj:spectrum-symmetry}
$D_\ell=U_\ell$ for all $\ell$.
\end{conjecture}

In this paper, we prove the proposed value of $D_3=U_3$ and show that Conjecture~\ref{conj:spectrum-symmetry} fails for $\ell=4$.
\begin{theorem}\label{thm:small-spectra}
$D_3=U_3=\{3,4,5,6,7,8,10,11,13\}$  and $D_4=U_4\setminus\{9\}=\{5,6,7,8,10,11,13,16,18\}$. Moreover, $D_4(9)=\varnothing$ and $U_4(9)=\{8\}$.
\end{theorem}

Chung, Graham and Spiro~\cite{CGS} observed that some gap values occur much more rarely than others. For example, the least element of $U_2(10)$ is $1917$. Thus they posed the following problem.
\begin{problem}[\cite{CGS}]
\label{prob:refined-densities}
Determine the densities of $D_\ell(m)$ and $U_\ell(m)$ for various
$\ell$ and $m$, and in particular for $\ell=1$.
\end{problem}
For a set $E\subseteq\mathbb N$, define its lower and upper densities by
\[
\underline d(E)
=\liminf_{X\to\infty}\frac{|E\cap[1,X]|}{X}
\qquad{and}\qquad
\overline d(E)
=\limsup_{X\to\infty}\frac{|E\cap[1,X]|}{X}.
\]
If these quantities are equal, their common value is the natural density
$d(E)$. The logarithmic density of $E$ is
\[
\delta_{\log}(E)
=\lim_{X\to\infty}\frac1{\log X}
\sum_{\substack{n\leq X\\ n\in E}}\frac1n.
\]
Natural density is the standard measure of asymptotic frequency. Logarithmic density is better suited to the multiplicative structure of the Fibonacci scales and may exist when natural density does not. For Problem~\ref{prob:refined-densities}, we prove the following asymptotic formula for the normalized counting functions.
\begin{theorem}\label{thm:phase-density} For $X\geq1$, let $t_*(X)$ be the largest even integer $t\geq2$ such that $F_t^2\leq X$, and let $\lambda_X=X/F_{t_*(X)}^2$. For every fixed $\ell,m\geq1$ and $\phi=(1+\sqrt5)/2$, there exists a continuous, multiplicatively $\phi^4$-periodic function $\mathcal P_{\ell,m}$ such that 
\[
 \frac{|D_\ell(m)\cap[1,X]|}{X}
 =\mathcal P_{\ell,m}(\lambda_X)+o(1)\qquad{and}\qquad
 \frac{|U_\ell(m)\cap[1,X]|}{X}
 =\mathcal P_{\ell,m}(\phi^2\lambda_X)+o(1).
\]
The two normalized counting functions have $\mathcal P_{\ell,m}([1,\phi^4])$ as their common set of subsequential limits. 
\end{theorem}

It follows that $D_\ell(m)$ and $U_\ell(m)$ have equal lower densities and equal upper densities. Their natural densities exist if and only if $\mathcal P_{\ell,m}$ is constant. Specifically, when they exist, they are equal. Their logarithmic densities always exist and are equal. For $\ell=1$, we obtain the exact logarithmic densities.
\begin{corollary}
\label{cor:first-density}
For $m\notin D_1$, $D_1(m)=U_1(m)=\varnothing$. For every $m\in D_1$, both sets have positive lower density but no
natural density, and their logarithmic densities are equal. Let $\mathfrak d_m=\delta_{\log}(D_1(m))=\delta_{\log}(U_1(m))$ and $\phi=(1+\sqrt5)/2$. Then
\[
\begin{array}{c|c}
m&4\mathfrak d_m\log\phi\\ \hline
1&2\log\sqrt5+(\sqrt5-5)\log\phi\\
2&-2\log\sqrt5+(11-3\sqrt5)\log\phi\\
3&-\log\sqrt5+\frac{5\sqrt5-7}{2}\log\phi\\
5&\log\sqrt5-\phi\log\phi
\end{array}.
\]
\end{corollary}

\section{Two preliminary lemmas}
Throughout the paper, let $\phi=(1+\sqrt5)/2$. For $s\geq2$, the Fibonacci recurrence shows that 
\begin{equation}\label{eq:fib-phi-identities}
\phi F_s+F_{s-1}=\phi^s,\quad
F_s-\phi F_{s-1}=(-1)^{s+1}\phi^{1-s}
\quad\text{and}\quad
\phi F_s-F_{s+1}=(-1)^{s+1}\phi^{-s}.
\end{equation}
Given a real number $x$, we write $\{x\}=x-\lfloor x\rfloor$. For $n\geq1$, let $\delta_n=\{\phi n\}=\{\phi^{-1}n\}$. If $n=aF_t+bF_{t-1}$ is the canonical representation of $n\geq2$, then define
\[
\eps_n=
\begin{cases}
\delta_n,&\text{if $t$ is even},\\
1-\delta_n,&\text{if $t$ is odd}.
\end{cases}
\]
Since $0<\phi^{-t}(\phi b-a)<1$ and $\phi n=aF_{t+1}+bF_t+(-1)^t\phi^{-t}(\phi b-a)$, we have $\eps_n=\phi^{-t}(\phi b-a)$. The first lemma determines the coefficients
associated with a fixed index.
\begin{lemma}\label{lem:candidate}
For $s,n\geq2$, define $e_s(n)=\{(-1)^s\phi n\}$ and $v_s(n)=\phi n+(-1)^{s+1}e_s(n)$.\\
$(i)$ The system 
\[
\begin{pmatrix}
F_s&F_{s-1}\\
F_{s+1}&F_s
\end{pmatrix}
\binom{a_s(n)}{b_s(n)}
=
\binom{n}{v_s(n)}
\]
has the unique integer solution $a_s(n)=\phi^{1-s}n-F_{s-1}e_s(n)$ and $b_s(n)=\phi^{-s}n+F_se_s(n)$. Moreover, $\tau(n)=s$ if and only if $1\leq a_s(n)\leq b_s(n)\leq F_s$.\\
$(ii)$ Let $q_s=F_{s-1}/F_s, \kappa_s=F_s\phi^{-s}$ and $\Tcal=\{(A,B):0\leq A\leq B\leq1\}$. For $A_s=a_s(n)/F_s$ and $B_s=b_s(n)/F_s$, define $L_s(A,B)=\bigl(A+q_sB,\ \kappa_s(\phi B-A)\bigr)$.
Then
\[
L_s(A_s,B_s)=\left(\frac{n}{F_s^2},e_s(n)\right).
\]
The maps $L_s$ converge uniformly on $\Tcal$ to the invertible map
\[
L(A,B)=\left(A+\phi^{-1}B,\frac{\phi B-A}{\sqrt5}\right).
\]
\end{lemma}

\begin{proof}
Since the determinant of the displayed matrix is $F_s^2-F_{s-1}F_{s+1}=(-1)^{s+1}$ and both $n$ and $v_s(n)=\phi n+(-1)^{s+1}e_s(n)$ are integers, the system has a unique integer solution. By (\ref{eq:fib-phi-identities}) and solving the equations $F_sa_s(n)+F_{s-1}b_s(n)=n$ and $F_{s+1}a_s(n)+F_sb_s(n)=v_s(n)$, we obtain $a_s(n)=\phi^{1-s}n-F_{s-1}e_s(n)$ and $b_s(n)=\phi^{-s}n+F_se_s(n)$.

Assume first that $\tau(n)=s$, and write $n=aF_s+bF_{s-1}$, where $1\leq a\leq b\leq F_s$. Let $\eta=\phi^{-s}(\phi b-a)$. Then $0<\eta<1$, since $0<\phi b-a\leq\phi F_s-1<\phi^s$. Moreover, $\phi n=aF_{s+1}+bF_s+(-1)^s\eta$. It follows that $v_s(n)=aF_{s+1}+bF_s$. Thus $(a,b)$ satisfies the
displayed system, and hence $a_s(n)=a$ and $b_s(n)=b$. Conversely, if $1\leq a_s(n)\leq b_s(n)\leq F_s$, then $n=a_s(n)F_s+b_s(n)F_{s-1}$ is the canonical representation of $n$. Hence $\tau(n)=s$.

The first coordinate of $L_s(A_s,B_s)$ is $A_s+q_sB_s=n/F_s^2$. The second follows from $\phi F_s+F_{s-1}=\phi^s$. Since $q_s$ converges to $\phi^{-1}$
and $\kappa_s$ converges to $1/\sqrt5$, the convergence is uniform on $\Tcal$. Finally, $\det L=(\phi+\phi^{-1})/\sqrt5=1$.
\end{proof}

Let $\alpha=1/(1+\phi^2)=\phi^{-1}/\sqrt5$, $\beta=\phi^2/(1+\phi^2)=\phi/\sqrt5$, $J=[\alpha,\beta]$ and $\Tcal^\circ=\{(A,B):0<A<B<1\}$. Since
\[
L^{-1}(\xi,e)=\left(\frac{\phi\xi}{\sqrt5}-\phi^{-1}e,
\frac{\xi}{\sqrt5}+e\right),
\]
$L^{-1}(\xi,e)\in\Tcal^\circ$ if and only if
\begin{equation}\label{eq:triangle-criterion}
\frac{\xi}{\sqrt5\phi^2}<e<
\min\left\{\frac{\phi^2\xi}{\sqrt5},
1-\frac{\xi}{\sqrt5}\right\}.
\end{equation}
The corresponding criterion for $L^{-1}(\xi,e)\in\Tcal$ is obtained by
replacing both strict inequalities with non-strict inequalities.
For $u>0$, choose the unique $r\in\mathbb Z$ such that $\xi=\phi^{4r}u\in[\phi^{-3},\phi)$, and define
\begin{equation}\label{eq:Theta}
\Theta(u)=\min\left\{\frac{\phi^2\xi}{\sqrt5},
1-\frac{\xi}{\sqrt5}\right\}.
\end{equation}
Equivalently,
\[
\Theta(u)=
\begin{cases}
\phi^2\xi/\sqrt5,&\phi^{-3}\leq\xi\leq\phi^{-1},\\
1-\xi/\sqrt5,&\phi^{-1}\leq\xi<\phi.
\end{cases}
\]
Thus $\Theta$ is continuous and has range $J$. Furthermore, it satisfies
$\Theta(\phi^4u)=\Theta(u)$ and
$\Theta(\phi^{-2}u)=1-\Theta(u)$.

\begin{lemma}\label{lem:layer-coding}
Let $(n_t)$ be a sequence of integers with $n_t\geq2$, indexed by integers $t\to\infty$ of fixed parity. Assume that $n_t/F_t^2\to u>0$. Let $\delta_{n_t}\to z$ when $t$ is even and $1-\delta_{n_t}\to z$ when $t$ is odd. Fix $h\in\mathbb Z$ and set $p_h=\{z+h\phi^{-1}\}$, with $p_h\notin\{0,\Theta(u)\}$. For all sufficiently large
$t$,
\begin{align}
t\text{ even}:\quad&
n_t+h\in D\quad\text{if and only if}\quad p_h<\Theta(u),\label{eq:layer-code-D}\\
t\text{ odd}:\quad&
n_t-h\in U\quad\text{if and only if}\quad p_h<\Theta(u).
\label{eq:layer-code-U}
\end{align}
\end{lemma}

\begin{proof}
Suppose first that $t$ is even. Since $\delta_{n_t+h}=\{\delta_{n_t}+h\phi^{-1}\}$ and $p_h\neq0$, we have $\delta_{n_t+h}\to p_h$. For every fixed $j\in\mathbb Z$, $(n_t+h)/F_{t+2j}^2\longrightarrow\phi^{-4j}u$. If $\tau(n_t+h)=t+2j$ along a subsequence, let $(A_t,B_t)\in\Tcal$ be the corresponding normalized canonical coefficients. We may assume that $(A_t,B_t)\to(A,B)\in\Tcal$. By Lemma~\ref{lem:candidate} and the uniform convergence $L_{t+2j}\to L$, $L(A,B)=(\phi^{-4j}u,p_h)$. Hence $L^{-1}(\phi^{-4j}u,p_h)\in\Tcal$. By the non-strict version of \eqref{eq:triangle-criterion}, we have 
\begin{equation}\label{eq:closed-candidate}
\frac{\phi^{-4j-2}u}{\sqrt5}\leq p_h\leq
\min\left\{\frac{\phi^{2-4j}u}{\sqrt5},
1-\frac{\phi^{-4j}u}{\sqrt5}\right\}.
\end{equation}
Choose $r\in\mathbb Z$ such that
$\xi=\phi^{4r}u\in[\phi^{-3},\phi)$, and set
$x_j=\phi^{-4j}u=\phi^{-4(j+r)}\xi$. Since $x_j/(\sqrt5\phi^2)\leq1-x_j/\sqrt5$ holds precisely when $x_j\leq\phi$, the interval in \eqref{eq:closed-candidate} is empty for $j<-r$, except when $j=-r-1$ and $\xi=\phi^{-3}$, where it reduces to $\{\alpha\}$. For $j=-r$, its upper endpoint is $\Theta(u)$. For $j>-r$, we have $x_j<\phi^{-3}$ and $p_h\leq\phi^2x_j/\sqrt5<\alpha\leq\Theta(u)$. Therefore, if $\tau(n_t+h)=t+2j$ for infinitely many $t$, then $p_h\leq\Theta(u)$. If $\tau(n_t+h)=t+2j+1$ for infinitely many $t$, the limiting normalized size and phase are $\phi^{-4j}(\phi^{-2}u)$ and $1-p_h$, respectively. By the same argument, with $\phi^{-2}u$ in place of $u$, we have $1-p_h\leq\Theta(\phi^{-2}u)=1-\Theta(u)$. Hence $p_h\geq\Theta(u)$.

It remains to prove that $k_t=\tau(n_t+h)-t$ is bounded. Suppose first that $k_t=-r_t$ on a subsequence, where $r_t\to\infty$. If $t-r_t$ is bounded, then $n_t+h\leq F_{t-r_t}F_{t-r_t+1}\leq2F_{t-r_t}^2$, contrary to the unboundedness of $n_t$. We may therefore assume that $t-r_t\to\infty$. For all sufficiently large $t$,
\[
\frac{n_t+h}{F_{t-r_t}^2}
\geq\frac{u}{2}\left(\frac{F_t}{F_{t-r_t}}\right)^2
\geq\frac{u}{2}F_{r_t+1}^2,
\]
again contrary to the canonical bound.

Suppose that $(k_t)$ is unbounded above, and choose a subsequence with $k_t\to\infty$. Set $s_t=t+k_t$. Let $n_t+h=a_tF_{s_t}+b_tF_{s_t-1}, A_t=\frac{a_t}{F_{s_t}}$ and $B_t=\frac{b_t}{F_{s_t}}$. Then $x_t=(n_t+h)/F_{s_t}^2=O\!\left(F_{k_t+1}^{-2}\right)=o(1)$. Since $x_t=A_t+q_{s_t}B_t$, where $A_t,B_t\geq0$ and $\inf_t q_{s_t}>0$, it follows that $A_t,B_t\to0$. So $e_{s_t}(n_t+h)=\kappa_{s_t}(\phi B_t-A_t)\longrightarrow0$. Restricting to one of the two parities of $s_t$ yields $e_{s_t}(n_t+h)\longrightarrow p_h$  or $e_{s_t}(n_t+h)\longrightarrow1-p_h$, which contradicts $0<p_h<1$. Hence $(k_t)$ is bounded above, and thus $(k_t)$ is bounded.

Assume that the equivalence in \eqref{eq:layer-code-D} fails infinitely often. Then there is a subsequence on which $k_t$ is constant and has the wrong
parity. The normalized canonical coefficients have a convergent subsequence in $\Tcal$ and its limit satisfies the closed triangle criterion. It follows that
\[
\begin{array}{ll}
k_t\equiv0\pmod2:& n_t+h\in D,\quad p_h\leq\Theta(u),\\
k_t\equiv1\pmod2:& n_t+h\notin D,\quad p_h\geq\Theta(u).
\end{array}
\]
Since $p_h\neq\Theta(u)$, these inequalities are strict in the required directions, contradicting our assumption. Hence \eqref{eq:layer-code-D} holds.

Now suppose that $t$ is odd. Set $z_t=1-\delta_{n_t}$. Since $p_h\neq0$, $1-\delta_{n_t-h}=\{z_t+h\phi^{-1}\}\longrightarrow p_h$. By applying the same argument to $n_t-h$, we obtain $n_t-h\in U$ if and only if $p_h<\Theta(u)$ for all sufficiently large $t$, as $\tau(n_t-h)-t$ is even precisely when $\tau(n_t-h)$ is odd. This proves \eqref{eq:layer-code-U}.
\end{proof}
\medskip
\section{Proof of Theorem \ref{thm:small-spectra}}
Let
\[
\begin{aligned}
\Gamma_1&=\{1,2,3,5\},&
\Gamma_2&=\{2,3,4,5,6,8,10\},\\
\Gamma_3&=\{3,4,5,6,7,8,10,11,13\},&
\Gamma_4&=\{5,6,7,8,10,11,13,16,18\}.
\end{aligned}
\]
For every $n\geq2$, we have $0<\eps_n\leq\phi^{-\tau(n)}\bigl(\phi F_{\tau(n)}-1\bigr)<\beta$, since $\beta=\phi/\sqrt5$ and $\beta\phi^t-(\phi F_t-1)=1+(-1)^t\phi^{1-t}/\sqrt5>0$ for $t\ge2$. Therefore,
\begin{equation}\label{eq:phase-bounds}
\begin{array}{ll}
n\in D:&0<\delta_n<\beta,\\
n\in U:&\alpha<\delta_n<1.
\end{array}
\end{equation}

\begin{lemma}\label{lem:exact-shift}
For $r\in\{5,6,7,8\}$, let $\Lambda_r=(\phi^{-1}+\phi^{2-r})/\sqrt5$. Suppose that $n=aF_t+bF_{t-1}$ is canonical, $t\geq r-1$, and
$\eps_n>\Lambda_r$. If
$N=n+(-1)^{t-r+1}F_{r-1}\geq2$, then $N$ has the same type as
$n$.
\end{lemma}

\begin{proof}
Suppose first that  $t=r-1$. If $a=b$, then $\eps_n=a\phi^{-r}\leq F_{r-1}\phi^{-r}<\Lambda_r$, where
\[
\Lambda_r-F_{r-1}\phi^{-r}
=\frac{\phi^{2-r}+(-1)^{r-1}\phi^{1-2r}}{\sqrt5}>0.
\]
This contradicts $\eps_n>\Lambda_r$. Hence $a<b$, and $N=n+F_t=(a+1)F_t+bF_{t-1}$ is canonical with index $t$.

Next suppose that $t\geq r$. We claim that $b-a\geq F_{t-r+2}$. Otherwise, $\eps_n\leq\phi^{-t}\bigl(\phi^{-1}F_t+F_{t-r+2}-1\bigr)$. Set $s=t-r+2$. Then
\[
\phi^{-t}\bigl(\phi^{-1}F_t+F_s-1\bigr)
=\Lambda_r-\phi^{-t}\left(
1+\frac{(-1)^t\phi^{-t-1}+(-1)^s\phi^{-s}}{\sqrt5}
\right).
\]
Since $t\geq r\geq5$ and $s\geq2$,
\[
1+\frac{(-1)^t\phi^{-t-1}+(-1)^s\phi^{-s}}{\sqrt5}
\geq
1-\frac{\phi^{-6}+\phi^{-2}}{\sqrt5}>0.
\]
It follows that $\eps_n<\Lambda_r$, a contradiction. Define $a'=a+F_{t-r}$ and $b'=b-F_{t-r+1}$. Since $b'-a'=b-a-F_{t-r+2}\geq0$, we have $1\leq a'\leq b'\leq F_t$. By the identity, $F_{t-r}F_t-F_{t-r+1}F_{t-1}=(-1)^{t-r+1}F_{r-1}$, $a'F_t+b'F_{t-1}=n+(-1)^{t-r+1}F_{r-1}=N$. Thus the lemma follows.
\end{proof}

\begin{lemma}\label{lem:finite-realization}
Let $\alpha<\theta<\beta$, $0<z<\theta$, and let $N\geq0$ be an integer. Assume that $\{z+j\phi^{-1}\}\notin\{0,\theta\}$ with $0\leq j\leq N$. Then there exist arbitrarily large integers $n_D$ and $n_U$ such that, for every $0\leq j\leq N$, the integers $n_D+j$ and $n_U+j$ belong to $D$ and $U$, respectively, precisely when $\{z+j\phi^{-1}\}<\theta$.
\end{lemma}
\begin{proof}
Since $\{z+j\phi^{-1}\}\notin\{0,\theta\}$ for $0\leq j\leq N$,
replacing $z$ by a sufficiently close point does not change which of the inequalities $\{z+j\phi^{-1}\}<\theta$ hold. We may therefore assume that neither $z$ nor
$\theta-z$ belongs to $\{\theta\phi^{-4j}:j\geq0\}$. Set $u=\sqrt5\theta/\phi^2$. Then $\phi^{-3}<u<\phi^{-1}$ and $\Theta(u)=\theta$. Choose $q,q'\geq0$ such that $\theta\phi^{-4q-4}<z<\theta\phi^{-4q}$ and $\theta\phi^{-4q'-4}<\theta-z<\theta\phi^{-4q'}$. Since $\theta<\beta$, both $1-\theta\phi^{-4q-2}$ and $1-\theta\phi^{-4q'-2}$ exceed $\theta$. Hence, by \eqref{eq:triangle-criterion}, we have $L^{-1}(\phi^{-4q}u,z), L^{-1}(\phi^{-4q'}u,\theta-z)\in\Tcal^\circ$. By Lemma~\ref{lem:candidate}, applied to coefficient grid approximations of these two points for even and odd $s$, respectively, we obtain canonical integers $x_s,y_s$ such that $x_s/F_{s-2q}^2\to u, \delta_{x_s}\to z, y_s/F_{s-2q'}^2\to u$ and $1-\delta_{y_s}\to\theta-z$. By Lemma~\ref{lem:layer-coding}, with $(t,h)=(s-2q,j)$ and
$(t,h)=(s-2q',-j)$, respectively, $x_s+j\in D$ precisely when
$\{z+j\phi^{-1}\}<\theta$, while $y_s+j\in U$ precisely when
$\{\theta-z-j\phi^{-1}\}<\theta$. These conditions are equivalent
because, for $p_j=\{z+j\phi^{-1}\}\notin\{0,\theta\}$,
$\{\theta-z-j\phi^{-1}\}=\{\theta-p_j\}<\theta$ precisely when
$p_j<\theta$. Since $x_s,y_s\to\infty$, the result follows.
\end{proof}

The following lemma determines the first return times of the rotation $x\mapsto\{x+\phi^{-1}\}$.
\begin{lemma}\label{lem:return-map}
For $\alpha<\theta<\beta$, let $R(x)=\{x+\phi^{-1}\}$ with $x\in[0,1)$, and let $r_\theta(x)$ be the least positive integer $j$ such that $R^j(x)\in[0,\theta)$. For $x\in[0,\theta)$, the following hold.
\begin{enumerate}
\item[(i)] If $\alpha<\theta\leq\phi^{-2}$, then
\[
r_\theta(x)=
\begin{cases}
2,&0\leq x<\theta-\phi^{-3},\\
5,&\theta-\phi^{-3}\leq x<\phi^{-4},\\
3,&\phi^{-4}\leq x<\theta.
\end{cases}
\]

\item[(ii)] If $\phi^{-2}<\theta\leq\phi^{-1}$, then
\[
r_\theta(x)=
\begin{cases}
2,&0\leq x<\theta-\phi^{-3},\\
3,&\theta-\phi^{-3}\leq x<\phi^{-2},\\
1,&\phi^{-2}\leq x<\theta.
\end{cases}
\]

\item[(iii)] If $\phi^{-1}<\theta<\beta$, then
\[
r_\theta(x)=
\begin{cases}
1,&x\in[0,\theta-\phi^{-1})\cup[\phi^{-2},\theta),\\
2,&\theta-\phi^{-1}\leq x<\phi^{-2}.
\end{cases}
\]
\end{enumerate}
\end{lemma}
\begin{proof}
Since $2\phi^{-1}=1+\phi^{-3}, 3\phi^{-1}=2-\phi^{-4}$ and $5\phi^{-1}=3+\phi^{-5}$,
we have $R^2(x)=x+\phi^{-3}$ and
\[R^3(x)=
\begin{cases}
x+1-\phi^{-4},&x<\phi^{-4},\\
x-\phi^{-4},&x\geq\phi^{-4}.
\end{cases}
\]

For {\rm (i)}, $R(x)\notin[0,\theta)$, while $R^2(x)\in[0,\theta)$ precisely when $x<\theta-\phi^{-3}$ and $R^3(x)\in[0,\theta)$ precisely when $x\geq\phi^{-4}$. On the remaining interval $[\theta-\phi^{-3},\phi^{-4})$,  we see that $R^4(x)=x+2\phi^{-3}>\theta$ and $R^5(x)=x+\phi^{-5}<\phi^{-3}<\theta$. For {\rm (ii)}, $R(x)\in[0,\theta)$ precisely when $x\geq\phi^{-2}$. For $x<\phi^{-2}$, the second iterate returns precisely when $x<\theta-\phi^{-3}$. The remaining interval lies in $[\phi^{-4},\phi^{-2})$, on which $R^3(x)=x-\phi^{-4}\in[0,\theta)$. For {\rm (iii)}, the first iterate returns precisely on $[0,\theta-\phi^{-1})\cup[\phi^{-2},\theta)$. On the complementary interval,
$R^2(x)=x+\phi^{-3}<\phi^{-1}<\theta$. Thus the statement holds.
\end{proof}

\begin{lemma}\label{lem:gap-blocks}
Each of the blocks $\mathbf b_A=(1,1,1,2,1,2,2), \mathbf b_B=(1,2,3,2,3,3)$ and $\mathbf b_C=(2,3,5,3,5,5)$ occurs among consecutive gaps of both $D$ and $U$. It follows that $\Gamma_3\subseteq D_3\cap U_3$ and $\Gamma_4\subseteq D_4\cap U_4$.
\end{lemma}
\begin{proof}
By Lemma~\ref{lem:return-map}, the partial sums of the successive return
times for
\[
(\theta,z)=\left(\frac23,\frac25\right),\qquad
\left(\frac37,\frac25\right),\qquad
\left(\frac27,\frac1{50}\right)
\]
are 
\[
0,1,2,3,5,6,8,10;\qquad
0,1,3,6,8,11,14;\qquad
0,2,5,10,13,18,23,
\]
respectively. Their consecutive differences are $\mathbf b_A,\mathbf b_B$, and $\mathbf b_C$. For each pair, $0<z<\theta$. Moreover, $\{z+j\phi^{-1}\}$ is irrational for $j\geq1$, whereas $0$ and $\theta$ are rational. By Lemma~\ref{lem:finite-realization}, each block
occurs among the consecutive gaps of both $D$ and $U$. The sums of three consecutive entries in these blocks are $\{3,4,5\}, \{6,7,8\}, \{10,11,13\}$, and the sums of four consecutive entries are $\{5,6,7\}, \{8,10,11\}, \{13,16,18\}$. This implies that $\Gamma_3\subseteq D_3\cap U_3$ and
$\Gamma_4\subseteq D_4\cap U_4$.
\end{proof}

For a fixed anchor $n$, write $\mathsf D_j$ when $n+j\in D$ and
$\mathsf U_j$ when $n+j\in U$. Consider
\[
\begin{array}{lll}
\mathsf P_1=(\mathsf D_0,\mathsf U_5,\mathsf U_8),&
\mathsf P_2=(\mathsf U_{-3},\mathsf U_0,\mathsf D_5),&
\mathsf P_3=(\mathsf U_0,\mathsf D_3,\mathsf D_5),\\
\mathsf P_4=(\mathsf D_0,\mathsf D_3,\mathsf U_8),&
\mathsf P_5=(\mathsf D_0,\mathsf U_3,\mathsf U_{13}),&
\mathsf P_6=(\mathsf U_0,\mathsf U_2,\mathsf D_{13}).
\end{array}
\]
Let $\operatorname{Occ}(\mathsf P)$ be the set of anchors at which
$\mathsf P$ occurs, with every displayed integer at least two.

\begin{lemma}\label{lem:exceptional-configurations}
The occurrence sets are $\operatorname{Occ}(\mathsf P_1)=\operatorname{Occ}(\mathsf P_4)=\{9,12\}, \operatorname{Occ}(\mathsf P_2)=\operatorname{Occ}(\mathsf P_5)=\varnothing, \operatorname{Occ}(\mathsf P_3)=\{4\}$ and $\operatorname{Occ}(\mathsf P_6)=\{20,25,30,35\}$.
\end{lemma}

\begin{proof}
Let $x=\delta_n$ and define $\gamma_r=\alpha+\phi^{1-r}$ for
$5\leq r\leq8$. Since $\delta_{n+j}=\{x+j\phi^{-1}\}$, the necessary
phase conditions are
\[
\begin{aligned}
\mathsf P_1:&\quad x<\beta,
 \ \{x+5\phi^{-1}\}>\alpha,
 \ \{x+8\phi^{-1}\}>\alpha,\\
\mathsf P_2:&\quad \{x-3\phi^{-1}\}>\alpha,
 \ x>\alpha,
 \ \{x+5\phi^{-1}\}<\beta,\\
\mathsf P_3:&\quad x>\alpha,
 \ \{x+3\phi^{-1}\}<\beta,
 \ \{x+5\phi^{-1}\}<\beta,\\
\mathsf P_4:&\quad x<\beta,
 \ \{x+3\phi^{-1}\}<\beta,
 \ \{x+8\phi^{-1}\}>\alpha,\\
\mathsf P_5:&\quad x<\beta,
 \ \{x+3\phi^{-1}\}>\alpha,
 \ \{x+13\phi^{-1}\}>\alpha,\\
\mathsf P_6:&\quad x>\alpha,
 \ \{x+2\phi^{-1}\}>\alpha,
 \ \{x+13\phi^{-1}\}<\beta.
\end{aligned}
\]
Clearly, $2\phi^{-1}=1+\phi^{-3}, 3\phi^{-1}=2-\phi^{-4}, 5\phi^{-1}=3+\phi^{-5}, 8\phi^{-1}=5-\phi^{-6}$ and $13\phi^{-1}=8+\phi^{-7}$. So we can
reduce these intersections to $x\in(\gamma_7,\beta)$ for $(\mathsf P_1,\mathsf P_4)$, $x\in(\alpha,1-\gamma_6)$ for $(\mathsf P_2,\mathsf P_3)$, $x\in(\gamma_5,\beta)$ for $\mathsf P_5$ and $x\in(\alpha,1-\gamma_8)$ for $\mathsf P_6$. Thus $\eps_n>\gamma_7,\gamma_6,\gamma_5,$ and $\gamma_8$ for
$(\mathsf P_1,\mathsf P_4)$, $(\mathsf P_2,\mathsf P_3)$,
$\mathsf P_5$, and $\mathsf P_6$, respectively.
More precisely,
\[
\gamma_5=\frac{20-8\sqrt5}{5},\quad
\gamma_6=\frac{12\sqrt5-25}{5},\quad
\gamma_7=\frac{95-41\sqrt5}{10},\quad
\gamma_8=\frac{32\sqrt5-70}{5}.
\]
Moreover,
\[
\begin{aligned}
\gamma_5-\Lambda_5&=(25-11\sqrt5)/10,&
\gamma_6-\Lambda_6&=(9\sqrt5-20)/5,\\
\gamma_7-\Lambda_7&=(65-29\sqrt5)/10,&
\gamma_8-\Lambda_8&=(47\sqrt5-105)/10,
\end{aligned}
\]
and each difference is positive. For $(\mathsf P_1,\mathsf P_4)$, $(\mathsf P_2,\mathsf P_3)$,
$\mathsf P_5$ and $\mathsf P_6$, the corresponding shifted integers are
$n+F_6=n+8$, $n+F_5=n+5$, $n+F_4=n+3$, and
$n+F_7=n+13$, respectively. Each configuration
requires the shifted integer to have the opposite type. By
Lemma~\ref{lem:exact-shift}, the corresponding anchor indices are less
than $6,5,4,$ and $7$, respectively.

For $(\mathsf P_1,\mathsf P_4)$, the anchor index is even and hence
belongs to $\{2,4\}$. For $t=2$, $\eps_n=\phi^{-3}<\gamma_7$. For
$t=4$, since $2\phi-1<\gamma_7\phi^4<3\phi-2$, the condition $\phi b-a>\gamma_7\phi^4$, where
$1\leq a\leq b\leq3$. It forces $b=3$ and $a\in\{1,2\}$. Hence
$n\in\{9,12\}$. The canonical representations
\[
9=F_4+3F_3,\quad 12=2F_4+3F_3,\quad 15=3F_4+3F_3,
\]
\[
14=F_5+3F_4,\quad17=F_5+4F_4,\quad20=F_5+5F_4
\]
show that both configurations occur at these two anchors.

For $(\mathsf P_2,\mathsf P_3)$, the anchor index is $t=3$. Since $2\phi-2<\gamma_6\phi^3<2\phi-1$, the condition $\phi b-a>\gamma_6\phi^3$, where
$1\leq a\leq b\leq2$, forces $(a,b)=(1,2)$ and $n=4$.
Since $n-3<2$, $\mathsf P_2$ does not occur. Moreover,
$4=F_3+2F_2$ is up, while $7=F_4+2F_3$ and
$9=F_4+3F_3$ are down. Thus $\mathsf P_3$ occurs precisely at $4$.
For $\mathsf P_5$, the only possible index is $t=2$, but
$\eps_n=\phi^{-3}<\gamma_5$. Hence $\mathsf P_5$ does not occur.

For $\mathsf P_6$, the anchor index is $t=3$ or $t=5$. Note that $2\phi-2<\gamma_8\phi^3<2\phi-1$ and $5\phi-5<\gamma_8\phi^5<4\phi-3$. Thus $t=3$ forces $n=4$, while $t=5$ leaves precisely $(a,b)=(1,3)$, $b=4$ with $1\leq a\leq3$ and $b=5$ with $1\leq a\leq4$. The candidate $n=4$ fails because
$n+13=17=F_5+4F_4$ is up. The first two alternatives for $t=5$ give
$n=14,17,22,27$, and
\[
27=3F_5+4F_4,\quad30=3F_5+5F_4,\quad
35=4F_5+5F_4,\quad40=5F_5+5F_4
\]
show that each corresponding integer $n+13$ is up. In the last
alternative, $n=5a+15$ with $1\leq a\leq4$. For $a\leq3$, $n+2=(a+1)F_5+4F_4$,
while $n+2=F_7+3F_6$ for $a=4$. Thus $n+2$ is up. Finally, $n+13=F_6+(a+4)F_5$
has canonical index six. Therefore,
$\operatorname{Occ}(\mathsf P_6)=\{20,25,30,35\}$.
\end{proof}

Since $D_1=U_1=\Gamma_1$ and $D_2=U_2=\Gamma_2$, every pair of
consecutive gaps $g_i,g_{i+1}$ of $D$ or $U$ satisfies
\begin{equation}\label{eq:gap-restrictions}
g_i,g_{i+1}\in\Gamma_1
\qquad\text{and}\qquad
g_i+g_{i+1}\in\Gamma_2.
\end{equation}
For $g_i,g_{i+1}\in\Gamma_1$, the condition
$g_i+g_{i+1}\in\Gamma_2$ fails only when
$(g_i,g_{i+1})=(2,5)$ or $(5,2)$.

For $X\in\{D,U\}$ and $\mathbf g=(g_1,\ldots,g_\ell)$, set
$m=g_1+\cdots+g_\ell$. Let $w_X(\mathbf g)$ be the word indexed by
$\{0,\ldots,m\}$ whose letters at $0,g_1,g_1+g_2,\ldots,g_1+\cdots+g_\ell=m$
are $X$, and whose remaining letters are the other element of
$\{D,U\}$. For $q\in\mathbb Z$, let $\mathsf P_i@q$ denote the
translate of $\mathsf P_i$ by $q$.

\begin{lemma}\label{lem:arithmetic-reduction}
Let $\mathbf g$ satisfy \eqref{eq:gap-restrictions}. If $\mathbf g$ has
three entries, then its weight belongs to
$\Gamma_3\cup\{9,15\}$. If $\mathbf g$ has four entries and the weights
of its two consecutive triples belong to $\Gamma_3$, then its weight
belongs to $\Gamma_4\cup\{4,9,12,14,15\}$.
\end{lemma}

\begin{proof}
For three gaps, $3\leq m\leq15$, and only $12$ and $14$ lie outside
$\Gamma_3\cup\{9,15\}$. Since
\[
5-g_i\in\{0,2,3,4\},\qquad
\sum_{i=1}^3(5-g_i)=15-m,
\]
the value $m=14$ is impossible, while $m=12$ forces the gaps to be
$2,5,5$ in some order. Every such order contains the adjacent pair
$(2,5)$ or $(5,2)$, contrary to \eqref{eq:gap-restrictions}.

For four gaps, the only values outside
$\Gamma_4\cup\{4,9,12,14,15\}$ are $17,19,$ and $20$. Since
\[
\sum_{i=1}^4(5-g_i)=20-m,
\]
the value $m=19$ is impossible. If $m=20$, all four gaps are $5$, so
each triple has weight $15\notin\Gamma_3$. If $m=17$, exactly one gap
is $2$ and the others are $5$. If $g_2=2$ or $g_3=2$, a forbidden
adjacent pair occurs; if $g_1=2$ or $g_4=2$, one consecutive triple
has weight $15$. Thus all three values are impossible.
\end{proof}

\begin{lemma}\label{lem:word-obstructions}
Let $\mathbf g$ satisfy \eqref{eq:gap-restrictions}.

\noindent
$(i)$ If $\mathbf g$ has three entries and weight $9$, then
$w_D(\mathbf g)$ contains $\mathsf P_1@0$ or $\mathsf P_2@4$, while
$w_U(\mathbf g)$ contains one of
$\mathsf P_2@3,\mathsf P_3@0,\mathsf P_4@1$.
If its weight is $15$, then $w_D(\mathbf g)$ contains
$\mathsf P_5@0$ and $w_U(\mathbf g)$ contains $\mathsf P_5@2$.

\noindent
$(ii)$ Suppose that $\mathbf g$ has four entries and the weights of its
two consecutive triples belong to $\Gamma_3$. If its weight is $9$, the
weight $9$ conclusions in $(i)$ remain valid. If its weight is $12$,
then $w_D(\mathbf g)$ contains one of $\mathsf P_1@0,\mathsf P_2@6,\mathsf P_2@7,\mathsf P_3@2$,
and $w_U(\mathbf g)$ contains one of $\mathsf P_3@0,\mathsf P_4@3,\mathsf P_4@4,\mathsf P_3@5$.
If its weight is $14$, then $w_D(\mathbf g)$ contains
$\mathsf P_1@0$ or $\mathsf P_2@9$, while $w_U(\mathbf g)$ contains
one of $\mathsf P_2@3,\mathsf P_3@0,\mathsf P_4@6$.
If its weight is $15$, then $w_D(\mathbf g)$ contains
$\mathsf P_5@0$ or $\mathsf P_6@2$, while $w_U(\mathbf g)$ contains
$\mathsf P_6@0$ or $\mathsf P_5@2$.
The occurrences of $\mathsf P_6@2$ in $w_D(\mathbf g)$ and
$\mathsf P_6@0$ in $w_U(\mathbf g)$ correspond to
$(5,5,3,2)$ and $(2,3,5,5)$, respectively.
\end{lemma}

\begin{proof}
Let $S=\{0,g_1,g_1+g_2,\ldots,m\}$. Suppose first that $m=9$. If $w_D(\mathbf g)$ contains neither
$\mathsf P_1@0$ nor $\mathsf P_2@4$, then $S\cap\{5,8\}\ne\varnothing$ and $S\cap\{1,4\}\ne\varnothing$.
With two internal partial sums, the pairs
$(1,5),(4,5),(4,8)$ force a gap of length $4$, while $(1,8)$ forces
a gap of length $7$. With three internal partial sums, each of the
first three pairs leaves two intervals of length $4$ and only one
unused partial sum. For $(1,8)$, the unused point must split the
interval of length $7$ into gaps $2,5$ or $5,2$, contrary to
\eqref{eq:gap-restrictions}. This proves the assertion for $w_D$.

If $w_U(\mathbf g)$ contains neither $\mathsf P_3@0$ nor
$\mathsf P_4@1$, then $S$ meets both $\{3,5\}$ and $\{1,4\}$.
With two internal partial sums, $3\in S$ forces $4\in S$ and hence
$\mathsf P_2@3$; if $3\notin S$, then $5\in S$, and the other point
forces a gap of length $4$. With three internal partial sums,
$3\in S$ and $8\notin S$ yield $\mathsf P_2@3$. If $3,8\in S$, the
third point is $1$ or $4$, which forces a forbidden pair or a gap of
length $4$. If $3\notin S$, the points $5$ and $1$ or $4$ leave two
intervals of length $4$ and only one unused partial sum. This proves
the assertion for $w_U$.

If $\mathbf g$ has three entries and $m=15$, then
$\mathbf g=(5,5,5)$ and $S=\{0,5,10,15\}$. Hence $w_D(\mathbf g)$
contains $\mathsf P_5@0$ and $w_U(\mathbf g)$ contains
$\mathsf P_5@2$.

For four entries, the two triple conditions require $m-g_1,m-g_4\in\Gamma_3$.
If $m=12$, then $g_1,g_4\in\{1,2,5\}$. The equation
$g_2+g_3=12-g_1-g_4$ and \eqref{eq:gap-restrictions} leave, up to
reversal,
\[
(1,x,6-x,5)\quad(x\in\{1,3,5\}),\qquad
(1,5,5,1),\qquad(2,2,3,5),\qquad(5,1,1,5).
\]
Every displayed vector except the last contains $\mathsf P_1@0$ in
$w_D$ and $\mathsf P_3@0$ in $w_U$. Their reversals contain one of
$\mathsf P_2@6,\mathsf P_2@7$ in $w_D$ and one of
$\mathsf P_4@3,\mathsf P_4@4$ in $w_U$. The last vector contains
$\mathsf P_3@2$ in $w_D$ and $\mathsf P_3@5$ in $w_U$.

If $m=14$, then $g_1,g_4\in\{1,3\}$, and the same conditions leave,
up to reversal,
\[
(1,5,5,3),\qquad(3,3,5,3).
\]
Both vectors contain $\mathsf P_1@0$ in $w_D$; in $w_U$ they contain
$\mathsf P_3@0$ and $\mathsf P_2@3$, respectively. Their reversals
contain $\mathsf P_2@9$ in $w_D$ and $\mathsf P_4@6$ in $w_U$.

If $m=15$, then $g_1,g_4\in\{2,5\}$, and
\eqref{eq:gap-restrictions} leaves only $(2,3,5,5)$ and
$(5,5,3,2)$. The former contains $\mathsf P_5@0$ in $w_D$ and
$\mathsf P_6@0$ in $w_U$; the latter contains $\mathsf P_6@2$ in
$w_D$ and $\mathsf P_5@2$ in $w_U$. Thus the result  follows. 
\end{proof}

We shall use the following type assignments:
\begin{equation}\label{eq:small-types}
\begin{aligned}
&\{3,4,6\}\subset U,\qquad
\{5,7,9,10,12,15\}\subset D,\\
&\{13,18,23,26,31,36,41\}\subset D,\qquad
\{5a+3b:1\leq a\leq b\leq5\}\subset U,\\
&\{21,37,42\}\subset U.
\end{aligned}
\end{equation}
These inclusions follow from
\[
\{3,4,6\}
=\{aF_3+bF_2:1\leq a\leq b\leq2\},
\]
\[
\{5,7,9,10,12,15\}
=\{aF_4+bF_3:1\leq a\leq b\leq3\},
\]
\[
\{13,18,23,26,31,36,41\}
=\{F_6+bF_5:1\leq b\leq3\}
 \cup\{2F_6+bF_5:2\leq b\leq5\},
\]
\[
\{5a+3b:1\leq a\leq b\leq5\}
=\{aF_5+bF_4:1\leq a\leq b\leq5\},
\]
and
\[
\{21,37,42\}
=\{F_7+F_6,F_7+3F_6,2F_7+2F_6\}.
\]
The corresponding canonical indices are $3,4,6,5,$ and $7$,
respectively.

\begin{proof}[\bf Proof of Theorem~\ref{thm:small-spectra}]
Let $n$ be the initial point of three consecutive gaps, let $\mathbf g$
be their gap vector, and let $m$ be its weight. By Lemma~\ref{lem:arithmetic-reduction}, $m\in\Gamma_3\cup\{9,15\}$.
A three gap word contains exactly four integers of its type.

Suppose that $m=9$ and $n\in D$. By
Lemma~\ref{lem:word-obstructions} and
Lemma~\ref{lem:exceptional-configurations}, $w_D(\mathbf g)$ contains
$\mathsf P_1@0$ and $n\in\{9,12\}$. If $n=12$, then the endpoint
$21$ belongs to $U$. If $n=9$, then $\{9,10,12,13,15\}\subseteq D\cap[9,18]$,
which contains more than four down integers.

Suppose that $m=9$ and $n\in U$. By
Lemma~\ref{lem:word-obstructions} and
Lemma~\ref{lem:exceptional-configurations}, $w_U(\mathbf g)$ contains
$\mathsf P_3@0$ or $\mathsf P_4@1$. The first possibility requires
$n=4$, but its endpoint $13$ belongs to $D$. The second requires
$n\in\{8,11\}$, and $\{8,11,14,16,17\}\subseteq U\cap[8,17], \{11,14,16,17,19\}\subseteq U\cap[11,20]$. Thus $9\notin D_3\cup U_3$.

If $m=15$, then, by Lemma~\ref{lem:word-obstructions}, the corresponding
word contains a translate of $\mathsf P_5$. By
Lemma~\ref{lem:exceptional-configurations},
$\operatorname{Occ}(\mathsf P_5)=\varnothing$. Hence $D_3\cup U_3\subseteq\Gamma_3$.
By Lemma~\ref{lem:gap-blocks},
$\Gamma_3\subseteq D_3\cap U_3$. Therefore, $D_3=U_3=\Gamma_3$.

Now let $n$ be the initial point of four consecutive gaps, with gap
vector $\mathbf g$ and weight $m$. The weights of both consecutive
triples belong to $\Gamma_3$. By Lemma~\ref{lem:arithmetic-reduction}, $m\in\Gamma_4\cup\{4,9,12,14,15\}$.
A four gap word contains exactly five integers of its type.

If $m=4$, all four gaps equal $1$, so five consecutive integers have
the same type. Since $2\in D$, $3\in U$, and both sets are infinite,
this block lies between two consecutive integers of the other type
whose difference is at least $6$, contrary to
$D_1=U_1=\Gamma_1$.

Suppose that $m=9$ and $n\in D$. By
Lemma~\ref{lem:word-obstructions} and
Lemma~\ref{lem:exceptional-configurations}, $w_D(\mathbf g)$ contains
$\mathsf P_1@0$ and $n\in\{9,12\}$. For $n=12$, the endpoint $21$
belongs to $U$; for $n=9$, $\{9,10,12,13,15,18\}\subseteq D\cap[9,18]$.
Thus $D_4(9)=\varnothing$. If $n\in U$, then, by
Lemma~\ref{lem:word-obstructions} and
Lemma~\ref{lem:exceptional-configurations}, $w_U(\mathbf g)$ contains
$\mathsf P_3@0$ or $\mathsf P_4@1$. The first possibility requires
$n=4$ and has the down endpoint $13$. The second requires
$n\in\{8,11\}$. Now $\{11,14,16,17,19,20\}\subseteq U\cap[11,20]$,
whereas $U\cap[8,17]=\{8,11,14,16,17\}$.
Hence $U_4(9)=\{8\}$.

Suppose that $m=12$. By Lemma~\ref{lem:word-obstructions} and
Lemma~\ref{lem:exceptional-configurations}, the only possibilities for
$w_D(\mathbf g)$ are $\mathsf P_1@0$ and $\mathsf P_3@2$. In the
first case, $n\in\{9,12\}$ and the endpoints $21,24$ belong to $U$.
In the second case, $n=2$ and the endpoint $14$ belongs to $U$.
For $w_U(\mathbf g)$, the same two lemmas leave $\mathsf P_3@0,\mathsf P_3@5,\mathsf P_4@3$ and $\mathsf P_4@4$.
The first requires $n=4$, but $4,6,8,11,14,16$ are all up. The second
requires $n=-1$. The third requires $n=6$ or $9$; the former has the
down endpoint $18$, and the latter belongs to $D$. The fourth requires
$n=5$ or $8$; the former belongs to $D$, while $\{8,11,14,16,17,19\}\subseteq U\cap[8,20]$.
Thus $12\notin D_4\cup U_4$.

Suppose that $m=14$. By Lemma~\ref{lem:word-obstructions} and
Lemma~\ref{lem:exceptional-configurations}, a down word contains
$\mathsf P_1@0$ and therefore starts at $9$ or $12$. The corresponding
intervals contain $\{9,10,12,13,15,18\}\subseteq D$ and $\{12,13,15,18,23,26\}\subseteq D$, respectively.
For an up word, the same two lemmas leave $\mathsf P_3@0$ or
$\mathsf P_4@6$. The first starts at $4$ and has the down endpoint
$18$. The second starts at $3$ or $6$, and the corresponding intervals
contain $\{3,4,6,8,11,14\}\subseteq U$ and $\{6,8,11,14,16,17\}\subseteq U$, respectively.
Thus $14\notin D_4\cup U_4$.

Suppose that $m=15$ and $n\in U$. By
Lemma~\ref{lem:word-obstructions} and
Lemma~\ref{lem:exceptional-configurations}, the gap vector is
$(2,3,5,5)$ and $n\in\{20,25,30,35\}$.
By \eqref{eq:small-types}, $\{n+7:n\in\{20,25,30,35\}\}=\{27,32,37,42\}\subseteq U$.
However, $7$ is not a partial sum of $(2,3,5,5)$. If $n\in D$, the
same two lemmas require the gap vector $(5,5,3,2)$ and $n+2\in\{20,25,30,35\}$.
By \eqref{eq:small-types}, $\{n+8:n+2\in\{20,25,30,35\}\}=\{26,31,36,41\}\subseteq D$,
although $8$ is not a partial sum of $(5,5,3,2)$. Hence
$15\notin D_4\cup U_4$.

By Lemma~\ref{lem:gap-blocks},
$\Gamma_4\subseteq D_4\cap U_4$. Therefore, $D_4=\Gamma_4, U_4=\Gamma_4\cup\{9\}, D_4(9)=\varnothing$ and $U_4(9)=\{8\}$.
\end{proof}

\section{Proofs of Theorem \ref{thm:phase-density} and Corollary \ref{cor:first-density}}\label{sec:density}

For a Lebesgue measurable set $A\subseteq[0,1)$, let
$\operatorname{Leb}(A)$ denote its Lebesgue measure and let
$\boldsymbol1_A$ denote its indicator function. For integers $\ell,m\geq1$, $0\leq\theta\leq1$, and
$0\leq h\leq m$, define $I_h(\theta)=\{z\in[0,1):\{z+h\phi^{-1}\}<\theta\}$. Let
\[
 g_{\ell,m}(\theta)=\operatorname{Leb}\left\{z:
 z\in I_0(\theta)\cap I_m(\theta),\quad
 \sum_{h=0}^m\boldsymbol1_{I_h(\theta)}(z)=\ell+1\right\}.
\]
The phase profile is
\begin{equation}\label{eq:phase-profile}
 \mathcal P_{\ell,m}(\lambda)
 =\frac1\lambda\int_0^\lambda g_{\ell,m}(\Theta(v))\,dv
 \qquad(\lambda>0),
\end{equation}
where the value of the integrand at $v=0$ is immaterial.  Since the event
can change only on
$\bigcup_{h=0}^m(I_h(\theta)\mathbin\triangle I_h(\theta'))$, $|g_{\ell,m}(\theta)-g_{\ell,m}(\theta')|\le(m+1)|\theta-\theta'|$.
For $(x,y)\in\Tcal$, define $u=x+\phi^{-1} y$ and $z=(\phi y-x)/\sqrt5$.

\begin{proof}[\bf Proof of Theorem~\ref{thm:phase-density}]
Let $t$ tend to infinity through even integers, let
$n_t=a_tF_t+b_tF_{t-1}\in\Scal_t$, and suppose that
$(a_t/F_t,b_t/F_t)\to(x,y)\in\Tcal^\circ$. By
Lemma~\ref{lem:candidate},
$(n_t/F_t^2,\delta_{n_t})=L_t(a_t/F_t,b_t/F_t)$ tends to
$L(x,y)=(u,z)$.
Assume that $\{z+h\phi^{-1}\}\notin\{0,\Theta(u)\}$ for
$0\leq h\leq m$. By Lemma~\ref{lem:layer-coding}, for all sufficiently
large $t$, $n_t+h\in D$ if and only if $\{z+h\phi^{-1}\}<\Theta(u)$ with $0\leq h\leq m$.

Let $\Omega_{\ell,m}\subset\Tcal^\circ$ consist of the points $(x,y)$
for which $c_h(x,y)=\boldsymbol1_{[0,\Theta(u))}\bigl(\{z+h\phi^{-1}\}\bigr)$ with $0\leq h\leq m$, satisfy $c_0=c_m=1$ and $\sum_{h=0}^m c_h=\ell+1$. Define $V_{\ell,m}(s)=\operatorname{area}\{(x,y)\in\Omega_{\ell,m}:u\leq s\}$.
Fix $\eta>0$. On $\Tcal\cap\{u\geq\eta\}$, the branch boundaries of
$\Theta$ and the exceptional sets $\{(x,y):\{z+h\phi^{-1}\}\in\{0,\Theta(u)\}\}$ with $0\leq h\leq m$, lie in a finite union $\mathcal L_\eta$ of line segments. Hence $\partial\Omega_{\ell,m}\cap\{u\geq\eta\}\subseteq\mathcal L_\eta\cup\partial\Tcal$.
Since the region $u<\eta$ has area $O(\eta^2)$,
$\Omega_{\ell,m}$ is Jordan measurable.

For fixed $s$, remove $\{u<\eta\}$ and the $\delta$-neighborhoods of
$\mathcal L_\eta\cup\partial\Tcal\cup\{u=s\}$. If the classification are not uniform on the remaining compact set, there would be even
integers $t_i\to\infty$, convergent coefficient grid points, and a fixed
$h$ for which it fails. By Lemma~\ref{lem:candidate}, the corresponding
scale and phase have the required limits. By
Lemma~\ref{lem:layer-coding}, this is impossible. The removed grid
proportion has limsup $O_\eta(\delta)+O(\eta^2)$. Since
$n/F_t^2=A+q_tB$ and $q_t\to\phi^{-1}$, the removal of $u=s$ also
accounts for the cutoff $n\leq sF_t^2$. Letting first $t\to\infty$ through even integers, then
$\delta\to0$, and finally $\eta\to0$, we obtain, for every fixed
$s\geq0$,
\[
\frac1{F_t^2}\#\{n\in\Scal_t:n\leq sF_t^2,\ 
n\in D_\ell(m)\}\longrightarrow V_{\ell,m}(s).
\]
Moreover, $0\leq V_{\ell,m}(s)\leq\phi^2s^2/2$, and $V_{\ell,m}$ is
continuous because every line $u=s$ has area zero. The normalized
counts and $V_{\ell,m}$ are nondecreasing in $s$ and constant for
$s\geq2$. A finite partition of $[0,2]$ on which every increment of
$V_{\ell,m}$ is less than $\varepsilon$ shows that the convergence is
uniform in $s$. For odd $t$, one has $e_t(n)=1-\delta_n$. By
Lemma~\ref{lem:candidate} and Lemma~\ref{lem:layer-coding}, $c_h$
records membership of $n-h$ in $U$. The same argument therefore gives
uniform convergence for counts of the right endpoints of occurrences
in $U_\ell(m)$.

Let $e=t_*(X)$ and write $t=e+2j$. For fixed $H$, uniformity in $s$ and
$F_{e+2j}^2/F_e^2=\phi^{4j}+o(1)$ permit termwise passage to the limit
for $|j|\leq H$. Since $1\leq\lambda_X<\phi^4+o(1)$ and
$|\Scal_t|=F_t(F_t+1)/2$, the negative tail is
$O(\sum_{j<-H}\phi^{4j})$. For $j>H$, put
$M_t=\lfloor X/F_{t-1}\rfloor$. Then
\[
 \#(\Scal_t\cap[1,X])\leq\tfrac12M_t(M_t+1),
\]
so the quadratic terms contribute $O(\sum_{j>H}\phi^{-4j})$, while
$\sum_{j>H}M_t/X\leq\sum_{j>H}F_{t-1}^{-1}=o(1)$. Both estimates are
uniform in $\lambda_X$. Letting first $X\to\infty$ and then
$H\to\infty$, we obtain
\[
 \frac{|D_\ell(m)\cap[1,X]|}{X}
 =\frac1{\lambda_X}\sum_{j\in\mathbb Z}
 \phi^{4j}V_{\ell,m}(\lambda_X\phi^{-4j})+o(1).
\]

The determinant one change of variables above has vertical section
\[
 \mathcal J(u)=\left(\frac{\phi^{-2}u}{\sqrt5},
 \min\left\{\frac{\phi^2u}{\sqrt5},1-\frac u{\sqrt5}\right\}\right)
 \quad(0<u<\phi),
\]
and $\mathcal J(u)=\varnothing$ otherwise. After setting
$v=\phi^{4j}u$, the $j$th term integrates over
$\mathcal J(v\phi^{-4j})$. For fixed $v$, these intervals are pairwise disjoint, and their union
agrees with $(0,\Theta(v))$ up to their endpoints. Indeed, if
$\xi=\phi^{4r}v\in[\phi^{-3},\phi)$, the interval for $j=-r$ has upper
endpoint $\Theta(v)$, those for $j<-r$ are empty, the remaining
intervals meet successively at their endpoints, and their lower
endpoints tend to zero. Since $\Theta(v\phi^{-4j})=\Theta(v)$ and the
event contains $z\in I_0(\Theta(v))$,
\[
 \sum_{j\in\mathbb Z}\phi^{4j}V_{\ell,m}(\lambda\phi^{-4j})
 =\int_0^\lambda g_{\ell,m}(\Theta(v))\,dv.
\]
By \eqref{eq:phase-profile}, this proves the required asymptotic formula
for $D_\ell(m)$.

The Lipschitz bound for $g_{\ell,m}$ and continuity of $\Theta$ show
that $\mathcal P_{\ell,m}$ is continuous. Substitution $v=\phi^4w$ in
\eqref{eq:phase-profile} proves its multiplicative $\phi^4$ periodicity.
For the odd layers, use $e-1$ as the base index. The same sum and tail
estimates give $\mathcal P_{\ell,m}(X/F_{e-1}^2)+o(1)$ for the
normalized right-endpoint count. After summation, replacing right
endpoints by left endpoints changes the count by at most $m$. Since
$X/F_{e-1}^2=\phi^2\lambda_X+o(1)$,
\[
 \frac{|U_\ell(m)\cap[1,X]|}{X}
 =\mathcal P_{\ell,m}(\phi^2\lambda_X)+o(1).
\]

Finally,
\[
 1\leq\lambda_X<
 \frac{F_{t_*(X)+2}^2}{F_{t_*(X)}^2}=\phi^4+o(1).
\]
Every sequence $X_k\to\infty$ therefore has a subsequence on which
$\lambda_{X_k}$ converges in $[1,\phi^4]$. Conversely, for every
$\lambda\in[1,\phi^4)$, the sequence
$X_r=\lfloor\lambda F_{2r}^2\rfloor$ satisfies
$\lambda_{X_r}\to\lambda$. Thus the down limit set is
$\mathcal P_{\ell,m}([1,\phi^4])$. The up limit set is the same as
$\mathcal P_{\ell,m}([\phi^2,\phi^6])
=\mathcal P_{\ell,m}([1,\phi^4])$.
\end{proof}

\begin{proposition}\label{prop:finite-profile}
The restriction of $g_{\ell,m}$ to $J$ is piecewise affine, with
breakpoints in $\{\alpha,\beta\}\cup\bigl(J\cap\{\{r\phi^{-1}\}:|r|\leq m\}\bigr)$.
On a finite subdivision of one multiplicative period,
$\mathcal P_{\ell,m}$ has the form
$A\lambda+B+C/\lambda$ with algebraic coefficients. Its extrema on
that period are determined by the subdivision endpoints and at most
one interior critical point from each nonconstant piece.
\end{proposition}

\begin{proof}
Set
\[
\mathcal B_m=\{\alpha,\beta\}\cup
\bigl(J\cap\{\{r\phi^{-1}\}:-m\leq r\leq m\}\bigr)
=\{\alpha=b_0<b_1<\cdots<b_s=\beta\}.
\]
By the definition of $I_h(\theta)$, its boundary points on the circle
are $\{-h\phi^{-1}\}$ and $\{\theta-h\phi^{-1}\}$. By comparing these
points, their circular order can change only when $\theta=\{(h-k)\phi^{-1}\}$ with $0\leq h,k\leq m$;
see also \cite[\S2.1]{AlessandriBerthe}. Between consecutive points of
$\mathcal B_m$, the arcs determined by these endpoints have fixed
membership patterns and lengths affine in $\theta$. Therefore, $g_{\ell,m}(\theta)=A_i\theta+B_i$ with $A_i\in\mathbb Z, B_i\in\mathbb Q(\sqrt5)$,
on $(b_i,b_{i+1})$. By the preceding Lipschitz bound, these formulas
extend continuously to the endpoints.

On $[\phi^{-3},\phi^{-1}]$ and $[\phi^{-1},\phi]$, we have $\Theta(v)=\phi^2v/\sqrt5$ and $\Theta(v)=1-v/\sqrt5$,
and both branches map onto $J$. Hence the preimages of the $b_i$
determine a finite subdivision of $[\phi^{-3},\phi]$ on which
$f_{\ell,m}=g_{\ell,m}\circ\Theta$ is affine with coefficients in
$\mathbb Q(\sqrt5)$. By the periodicity
$f_{\ell,m}(\phi^4v)=f_{\ell,m}(v)$,
\[
\begin{aligned}
\int_0^\lambda f_{\ell,m}(v)\,dv
&=\sum_{r=0}^{\infty}
  \int_{\phi^{-4(r+1)}\lambda}^{\phi^{-4r}\lambda}
  f_{\ell,m}(v)\,dv\\
&=\sum_{r=0}^{\infty}\phi^{-4r}
  \int_{\phi^{-4}\lambda}^{\lambda}f_{\ell,m}(v)\,dv\\
&=\frac1{1-\phi^{-4}}
  \int_{\phi^{-4}\lambda}^{\lambda}f_{\ell,m}(v)\,dv.
\end{aligned}
\]
By \eqref{eq:phase-profile}, we have
\begin{equation}\label{eq:cycle-average}
\mathcal P_{\ell,m}(\lambda)
=\frac1{\lambda(1-\phi^{-4})}
\int_{\phi^{-4}\lambda}^{\lambda}f_{\ell,m}(v)\,dv.
\end{equation}

On a piece where $f_{\ell,m}(\lambda)=p\lambda+q$, differentiation of
\eqref{eq:phase-profile} yields $\bigl(\lambda\mathcal P_{\ell,m}(\lambda)\bigr)'=f_{\ell,m}(\lambda)$,
and hence $\mathcal P_{\ell,m}(\lambda)=p\lambda/2+q+C/\lambda$.
All subdivision endpoints and the coefficients $p,q$ belong to
$\mathbb Q(\sqrt5)$. By \eqref{eq:cycle-average}, evaluation at an
endpoint shows that $C\in\mathbb Q(\sqrt5)$. Finally, $\mathcal P_{\ell,m}'(\lambda)=p/2-C/\lambda^2$, so each nonconstant piece has at most one interior critical point, while the endpoints suffice for a constant piece.
\end{proof}

\begin{corollary}\label{cor:density-consequences}
The lower densities of $D_\ell(m)$ and $U_\ell(m)$ are both the minimum
of $\mathcal P_{\ell,m}$ on $[1,\phi^4]$, and their upper densities are
both its maximum. They have natural densities if and only if
$g_{\ell,m}$ is constant on $J$, in which case the densities are equal.
Their common lower density is positive if and only if
$g_{\ell,m}\not\equiv0$ on $J$. Their logarithmic densities always
exist and are both equal to
\begin{equation}\label{eq:log-density}
 \frac1{4\log\phi}\int_\alpha^\beta
 \frac{g_{\ell,m}(\theta)}{\theta(1-\theta)}\,d\theta.
\end{equation}
\end{corollary}

\begin{proof}
Let $f_{\ell,m}=g_{\ell,m}\circ\Theta$. By
Theorem~\ref{thm:phase-density}, both normalized counting functions
have $\mathcal P_{\ell,m}([1,\phi^4])$ as their set of subsequential
limits. By continuity and compactness, its minimum and maximum are the
common lower and upper densities. A natural density exists exactly when
this set is a singleton. By \eqref{eq:phase-profile},
$\mathcal P_{\ell,m}\equiv c$ implies $f_{\ell,m}\equiv c$ after
differentiation, while the converse follows directly from the integral.
Since $\Theta((0,\infty))=J$, this is equivalent to constancy of
$g_{\ell,m}$ on $J$.

By definition, $g_{\ell,m}\geq0$. If $g_{\ell,m}\not\equiv0$, then by
continuity $f_{\ell,m}$ is positive on a subinterval of every
multiplicative period. By \eqref{eq:cycle-average},
$\mathcal P_{\ell,m}(\lambda)>0$ for every $\lambda>0$. By compactness,
its minimum on $[1,\phi^4]$ is positive. If
$g_{\ell,m}\equiv0$, then by \eqref{eq:phase-profile},
$\mathcal P_{\ell,m}\equiv0$.

For $t\geq1$, let $\lambda_t=t/F_{t_*(t)}^2$.
For $E\in\{D_\ell(m),U_\ell(m)\}$, let $A_E(X)=|E\cap[1,X]|$ and
\[
 Q_E(\lambda)=
 \begin{cases}
 \mathcal P_{\ell,m}(\lambda),&E=D_\ell(m),\\
 \mathcal P_{\ell,m}(\phi^2\lambda),&E=U_\ell(m).
 \end{cases}
\]
By Theorem~\ref{thm:phase-density},
$A_E(t)/t=Q_E(\lambda_t)+e_E(t)$, where $e_E(t)=o(1)$. By partial
summation,
\[
 \sum_{\substack{n\leq X\\n\in E}}\frac1n
 =\frac{A_E(X)}X+\int_1^X\frac{A_E(t)}{t^2}\,dt.
\]
The contribution of $e_E$ is $o(\log X)$, as follows by splitting the
integral at $T_\varepsilon$. Since $F_{2r}=(\phi^{2r}-\phi^{-2r})/\sqrt5$, we have
\[
F_{2r}^2=\frac{\phi^{4r}}5(1-\phi^{-4r})^2,\qquad
R_r:=\frac{F_{2r+2}^2}{F_{2r}^2}\to\phi^4,\qquad
\log F_{2r}^2=4r\log\phi+O(1).
\]
By continuity and multiplicative $\phi^4$ periodicity,
\[
 \int_{F_{2r}^2}^{F_{2r+2}^2}
 Q_E\!\left(\frac{t}{F_{2r}^2}\right)\frac{dt}{t}
 =\int_1^{R_r}Q_E(\lambda)\frac{d\lambda}{\lambda}
 =I+o(1),
\]
where
\[
 I=\int_1^{\phi^4}\mathcal P_{\ell,m}(\lambda)
 \frac{d\lambda}{\lambda}.
\]
Thus $N$ complete phase intervals contribute $NI+o(N)$, their number
below $X$ is $(\log X)/(4\log\phi)+O(1)$, and the final partial interval
contributes $O(1)$. Since $A_E(X)/X=O(1)$, partial summation shows that
both logarithmic densities equal $I/(4\log\phi)$.

By the $4\log\phi$ periodicity of
$s\mapsto\mathcal P_{\ell,m}(e^s)$,
\[
 I=\int_{\phi^{-3}}^\phi\mathcal P_{\ell,m}(v)\frac{dv}{v}.
\]
By \eqref{eq:phase-profile},
$(v\mathcal P_{\ell,m}(v))'=f_{\ell,m}(v)$. Since
$\mathcal P_{\ell,m}(\phi)=\mathcal P_{\ell,m}(\phi^{-3})$,
\[
 I=\int_{\phi^{-3}}^\phi f_{\ell,m}(v)\frac{dv}{v}.
\]
By the substitutions $\theta=\phi^2v/\sqrt5$ and
$\theta=1-v/\sqrt5$ on the two branches of $\Theta$,
\[
\begin{aligned}
 I&=\int_\alpha^\beta\frac{g_{\ell,m}(\theta)}{\theta}\,d\theta
   +\int_\alpha^\beta\frac{g_{\ell,m}(\theta)}{1-\theta}\,d\theta\\
  &=\int_\alpha^\beta
    \frac{g_{\ell,m}(\theta)}{\theta(1-\theta)}\,d\theta.
\end{aligned}
\]
This proves \eqref{eq:log-density}.
\end{proof}

\begin{proof}[\bf Proof of Corollary~\ref{cor:first-density}]
Since $D_1=U_1=\Gamma_1=\{1,2,3,5\}$, the first order classes are
empty for all other $m$. Let $g_m=g_{1,m}$. For
$A_\theta=[0,\theta)$ and $Rz=\{z+\phi^{-1}\}$, define
\[
 C_m(\theta)=A_\theta\cap R^{-m}A_\theta
 \cap\bigcap_{h=1}^{m-1}R^{-h}\bigl([0,1)\setminus A_\theta\bigr).
\]
By definition, $C_m(\theta)=\{z\in A_\theta:r_\theta(z)=m\}$ and
$g_m(\theta)=\operatorname{Leb}C_m(\theta)$. By
Lemma~\ref{lem:return-map}, the lengths of these first return cells are
\[
\begin{array}{c|cccc}
\theta&g_1(\theta)&g_2(\theta)&g_3(\theta)&g_5(\theta)\\ \hline
\alpha\leq\theta\leq\phi^{-2}
 &0&\theta-\phi^{-3}&\theta-\phi^{-4}&\phi^{-2}-\theta\\
\phi^{-2}\leq\theta\leq\phi^{-1}
 &\theta-\phi^{-2}&\theta-\phi^{-3}&\phi^{-1}-\theta&0\\
\phi^{-1}\leq\theta\leq\beta
 &2\theta-1&1-\theta&0&0.
\end{array}
\]
By the Lipschitz continuity of $g_m$, the formulas also hold at the
endpoints.

The table shows that each $g_m$, $m\in\{1,2,3,5\}$, is nonzero and
nonconstant on $J$. By Corollary~\ref{cor:density-consequences}, the
corresponding refined sets have positive lower density but no natural
density, and
\[
 4\mathfrak d_m\log\phi
 =\int_\alpha^\beta
 \frac{g_m(\theta)}{\theta(1-\theta)}\,d\theta.
\]
On every affine piece,
\[
 \int\frac{A\theta+B}{\theta(1-\theta)}\,d\theta
 =B\log\theta-(A+B)\log(1-\theta).
\]
By evaluation at $\alpha,\phi^{-2},\phi^{-1},\beta$, the four values
stated in the corollary follow.
\end{proof}

\section*{Acknowledgements}
\noindent
The authors declare that there is no conflict of competing interest.


\begin{thebibliography}{99}
\bibitem{Cohn} J.~H.~E. Cohn, Recurrent sequences including $N$, \emph{Fibonacci Quart.} 29 (1991), 30--36.

\bibitem{JonesKiss} J.~P. Jones and P.~Kiss, Representation of integers as terms of a linear recurrence with maximal index, \emph{Acta Acad. Paedagog. Agriensis Sect. Math. (N.S.)} 25 (1998), 21--37.

\bibitem{EnglundBicknell} D.~A. Englund and M.~Bicknell-Johnson, Maximal subscripts within generalized Fibonacci sequences, \emph{Fibonacci Quart.} 38 (2000), 104--113.

\bibitem{CGS} F.~Chung, R.~L. Graham, and S.~Spiro, Slow Fibonacci walks, \emph{J. Number Theory} 210 (2020), 142--170.

\bibitem{Spiro} S.~Spiro, Slow recurrences, \emph{J. Number Theory} 218 (2021), 370--401.

\bibitem{AlessandriBerthe} P.~Alessandri and V.~Berth\'e, Three distance theorems and combinatorics on words, \emph{Enseign. Math.} 44(2) (1998), 103--132.

\end{thebibliography}
\end{document}